\documentclass[english]{amsart}
\usepackage[T1]{fontenc}
\usepackage[latin9]{inputenc}
\usepackage{geometry}
\usepackage{amsmath}
\usepackage{amsthm}
\usepackage{amssymb}
\usepackage{mathtools}
\usepackage[all]{xy}

\usepackage{mathrsfs}
\usepackage{hyperref}
\usepackage{color}
\hypersetup{colorlinks=true,citecolor=blue}

\usepackage{babel}

\theoremstyle{plain}
\newtheorem{thm}{Theorem}[section]
\newtheorem{lemm}[thm]{Lemma}
\newtheorem{prop}[thm]{Proposition}
\newtheorem{cor}[thm]{Corollary}

\newtheorem{qq}[thm]{Question}

\theoremstyle{definition}
\newtheorem{defn}[thm]{Definition}
\newtheorem{example}[thm]{Example}
\newtheorem{construction}[thm]{Construction}

\theoremstyle{remark}
\newtheorem*{rmk}{Remark}
\newtheorem*{claim}{Claim}

\DeclareMathOperator{\Spec}{Spec}

\DeclareMathOperator{\Spf}{Spf}

\DeclareMathOperator{\Sp}{Sp}

\DeclareMathOperator{\Alb}{Alb}
\DeclareMathOperator{\Pic}{Pic}
\DeclareMathOperator{\Frac}{Frac}

\newcommand{\et}{\mathrm{\acute{e}t}}

\newcommand{\cO}{\mathcal{O}}

\begin{document}

\title{Line bundles on rigid varieties and Hodge symmetry}

\author{David Hansen and Shizhang Li}

\begin{abstract} We prove several related results on the low-degree Hodge
  numbers of proper smooth rigid analytic varieties over non-archimedean fields.
  Our arguments rely on known structure theorems for the relevant Picard
  varieties, together with recent advances in $p$-adic Hodge theory. We also
  define a rigid analytic Albanese naturally associated with any smooth proper
  rigid space.
\end{abstract}

\maketitle

\section{Introduction}

Let $K$ be a $p$-adic field, i.e.~a complete discretely valued extension of
$\mathbf{Q}_p$ with perfect residue field $\kappa$\footnote{Note that the
  perfectness assumption of the residue field here is not essential, as Hodge
  numbers doesn't change under ground field extension.}. Let $X$ be a smooth
proper rigid analytic space over $K$.
Among fundamental invariants associated with $X$ are its Hodge numbers
$h^{i,j} \coloneqq \dim_K H^j(\Omega_X^i)$.
 In this paper, we study the relationship
between the Hodge numbers $h^{1,0}$ and $h^{0,1}$ of $X$.

For a compact complex manifold $Y$, we always have $\dim H^1(Y,\cO_Y) \geq \dim
{H^0(Y,\Omega^1_Y)}^{d=0}$ (c.f.~\cite[Chapter IV, Section 2]{BHPV}). In the
rigid analytic setting, Scholze proved that the Hodge--de~Rham spectral sequence
always degenerates at $E_1$, and in particular every global 1-form on $X$ as
above is automatically closed (c.f.~\cite[Theorem 8.4]{Scholze}). One is
naturally led to guess that for $X$ as above we always have $\dim H^1(X,\cO_X)
\geq \dim H^0(X,\Omega^1_X)$. In this paper we confirm this inequality assuming
that $X$ has a strictly semistable formal model (c.f.~\cite[Section 1]{HL}) and
provide a geometric interpretation of the difference.

\begin{thm}[Main Theorem]
\label{Main Theorem}
Under the conditions stated above, we have
\[
  \dim H^1(X,\cO_X) \geq \dim H^0(X,\Omega^1_X).
\]
Moreover, the difference between two numbers above is the virtual torus rank of
the Picard variety of $X$ (to be defined in the next section).
\end{thm}

\begin{rmk}
\leavevmode
\begin{enumerate}
\item Possessing strictly semistable
  reduction is stable under finite \'{e}tale base extension, hence the theorem is
  insensitive to finite unramified extensions of the ground field. 
  Since the central fiber of a strictly semistable formal model is generically smooth,
  by Hensel's lemma, we may and do
  assume that $X$ has a $K$-rational point $x: \Sp(K) \to X$. We will fix this
  rational point from now on.
\item The proof relies crucially on the assumption that $X$ has a strictly
  semistable formal model, which we use to determine the structure of the Picard
  variety of $X$, c.f.~Theorem~\ref{structure of picard} below. We certainly
  expect that the structure of the Picard variety should be of this shape in
  general. However, it is also a long standing folklore conjecture that any
  quasi-compact smooth rigid space potentially admits a strictly semistable
  formal model.
\item Assuming a result in progress by Conrad--Gabber along with the semistable
  reduction conjecture above, the Main Theorem holds for any smooth proper rigid
  space over any complete non-archimedean extension of $\mathbf{Q}_p$.
\end{enumerate}
\end{rmk}

In a complementary direction, the second author~\cite{ProperPicard} singled out
the class of smooth proper rigid spaces admitting some formal model with
\emph{projective} special fiber. In particular, according to Theorem 1.1 of
loc.~cit., the Picard variety of any such $X$ is automatically proper. Combining
this with L\"{u}tkebohmert's structure theorem (c.f.~\cite{abeloid}) for smooth
proper rigid groups and the comparison results of~\cite{Scholze}, we deduce the
following result.

\begin{thm}
\label{Main Corollary}
Let $X$ be a smooth proper rigid space over a $p$-adic field $K$. Assume that
$X$ has a formal model $\mathcal{X}$ over $\Spf(\mathcal{O}_K)$ whose special
fiber is projective. Then we have
\[
  h^{1,0}(X)=h^{0,1}(X).
\]
\end{thm}

\begin{rmk}
In this Theorem, we do not need to assume that \(X\) has potentially
  semistable reduction.
\end{rmk}

This result suggests that the condition of admitting a formal model with
projective reduction could be a natural rigid analytic analogue of the
K\"{a}hler condition. In particular, it is natural to ask if this condition
implies Hodge symmetry in higher degrees:

\begin{qq}
  Let $X$ be a smooth proper rigid space admitting a formal model with
  projective reduction. Is it true that $h^{i,j}(X)=h^{j,i}(X)$ for all $i,j$?
\end{qq}

By combining Theorem 1.2 and Serre duality, it is easy to see that this question
has an affirmative answer for rigid analytic surfaces.

\section{Preliminaries}
\label{Preliminaries}

In this section, we record some preliminary results from~\cite{HL} which will be
used in the proof of the Main Theorem. We remark that these results hold for
\emph{arbitrary} discretely valued non-archimedean fields $K$ (not necessarily an
extension of $\mathbf{Q}_p$). Throughout this section, $X$ will be a smooth
proper rigid space over such a $K$.

In the paper~\cite{HL}, L\"{u}tkebohmert and Hartl considered the Picard functor
\[
  \underline{\Pic}_{X/K}:(\text{Smooth rigid spaces over } K) \to (\text{Sets}), V
  \mapsto \underline{\Pic}_{X/K}(V)
\]
where
\[
  \underline{\Pic}_{X/K}(V)=\{\mathrm{Isomclass}(\mathcal{L},\lambda):
  \mathcal{L} \text{ a line bundle on } X \times_K V, \lambda: \mathcal{O}_V
  \xrightarrow{\sim} {(x,id)}^*\mathcal{L} \text{ an isomorphism}\}.
\]

Let us summarize several main statements of the paper mentioned above.

\begin{thm}[Summary of Theorem 0.1, Proposition 3.13, Theorem 3.14 and Theorem
  3.15 of~\cite{HL}]
\label{structure of picard}
Assume $X$ has a strictly semistable formal model.
\begin{enumerate}
\item The functor above is represented by a smooth rigid group denoted as
  $\Pic_X$.
\end{enumerate}

After a suitable finite base extension we have:
\begin{enumerate}
  \setcounter{enumi}{1}
\item The identity component $\Pic^0_X$ of $\Pic_X$, i.e.~the \emph{Picard
    variety of $X$} in their terminology, canonically admits a Raynaud's
  uniformization:
  \[
    \xymatrix{
      & \Gamma \ar[d] & \\
      T \ar[r] & \hat{P} \ar[d] \ar[r] & B \\
      & \Pic^0_X }
  \]
  Here $T$ is a split torus of dimension $r$, $\Gamma$ is a lattice of rank $k
  (\leq r)$, and $B$ is a good reduction abeloid variety (c.f.~\cite{abeloid}),
  i.e. $B$ is the rigid generic fiber of a formal abelian scheme.
\item Non-canonically, $\Pic^0_X$ may be written as an extension of an abeloid
  variety by a split torus of dimension $r-k$\footnote{Although the way to write
    this extension is non-canonical, this number $r-k$ only depends on $X$}:
  \[
    0 \to \mathbb{G}_m^{r-k} \to \Pic^0_X \to A \to 0.
  \]
\item The geometric component group of $\Pic_X \times_{K}
  \mathbb{C}_K$\footnote{Here and elsewhere in this paper, we use $\mathbb{C}_K
    \coloneqq \widehat{\overline{K}}$ to denote the completion of an algebraic
    closure of $K$.}, i.e.~the \emph{N\'{e}ron--Severi group of $X$} in their
  terminology, is a finitely generated abelian group.
\end{enumerate}
\end{thm}

\begin{defn}
\label{virtual torus rank}
The \textit{virtual torus rank} of $\Pic^0_X$ is defined to be $r-k$ in the
notation above.
\end{defn}

It is easy to derive the following structural properties of the Tate module of
$\Pic_X$.

\begin{prop}
\label{structure of Tate module}
\leavevmode
\begin{enumerate}
\item The Tate module of $\Pic_X$ is the same as that of $\Pic^0_X$.
\item There are two canonical short exact sequences of $p$-adic $G_K$
  representations:
  \[
    0 \to V_p(T) \to V_p(\hat{P}) \to V_p(B) \to 0,
  \]
  \[
    0 \to V_p(\hat{P}) \to V_p(\Pic^0_X)=V_p(\Pic_X) \to
    \varprojlim(\Gamma/p^n\Gamma) \otimes_{\mathbb{Z}_p} \mathbb{Q}_p \to 0.
  \]
  Here $V_p(G)=\varprojlim G(\overline{K})[p^n] \otimes_{\mathbb{Z}_p} \mathbb{Q}_p$ denotes
  the rational $p$-adic Tate module associated with any commutative rigid
  analytic group $G$.
\item There is a non-canonical short exact sequence of $p$-adic $G_K$
  representations:
  \[
    0 \to {\mathbb{Q}_p(1)}^{r-k} \to V_p(\Pic^0_X) \to V_p(A) \to 0.
  \]
\end{enumerate}
\end{prop}

\begin{proof}
  (1) follows from Theorem~\ref{structure of picard} (4), while (2) and (3) are
  consequences of Theorem~\ref{structure of picard} (2) and (3), respectively.
\end{proof}

\section{Proof of the Main Theorem}
Now we specialize the results from Section~\ref{Preliminaries} to the situation
where $K$ is of mixed characteristic (i.e.~an extension of $\mathbb{Q}_p$). With
the aid of Proposition~\ref{structure of Tate module} and Hodge--Tate
comparison, it is easy to prove the Main Theorem.

\begin{proof}[Proof of Theorem~\ref{Main Theorem}]
  By Hodge--Tate comparison for smooth proper rigid spaces over $K$
  (c.f.~\cite[Theorem 7.11]{Scholze}), we have a canonical $G_K$-equivariant
  isomorphism
  \[
    H^1_{\et}(X_{\mathbb{C}_K},\mathbb{Q}_p(1)) \otimes_{\mathbb{Q}_p}
    \mathbb{C}_K=(H^1(X,\cO_X) \otimes_K \mathbb{C}_K(1)) \oplus (H^0(X,\Omega^1_X)
    \otimes_K \mathbb{C}_K),
  \]
  where $\mathbb{C}_K$ is the completion of an algebraic closure of $K$. On the
  other hand, by the usual Kummer sequence we have
  \[
    H^1_{\et}(X_{\mathbb{C}_K},\mathbb{Q}_p(1))=V_p(\Pic_X).
  \]
  Combining these isomorphisms with Hodge--Tate comparison for $A$ and the
  structural results for $V_p(\Pic_X)$ described in Proposition~\ref{structure
    of Tate module} (3), we see that
  \[
    \dim_K {(V_p(\Pic_X) \otimes_{\mathbb{Q}_p}
      \mathbb{C}_K)}^{G_K}=\dim_K{(V_p(A) \otimes_{\mathbb{Q}_p}
      \mathbb{C}_K)}^{G_K}=\dim_K(H^1(\hat{A},\cO_{\hat{A}}))=\dim(\hat{A})=\dim(A)
  \]
  and similarly
  \[
    \dim_K{(V_p(\Pic_X)\otimes_{\mathbb{Q}_p}
      \mathbb{C}_K(-1))}^{G_K}=r-k+\dim(A),
  \]
  where $r-k$ is the virtual torus rank of the Picard variety of $X$ as in
  Definition~\ref{virtual torus rank}. By Hodge--Tate comparison for $X$, the
  former is $h^{1,0}(X)$ and the latter is $h^{0,1}(X)$, so taking the
  difference gives $h^{0,1}(X)-h^{1,0}(X)=r-k$, as desired.
\end{proof}

One sees that the argument above only uses the qualitative structure of the
Picard variety. Similarly, it is easy to prove Theorem~\ref{Main Corollary}.

\begin{proof}[Proof of Theorem~\ref{Main Corollary}]
\cite[Theorem 1.1]{ProperPicard} says that in this situation the Picard
  variety is an abeloid variety. Therefore $r-k=0$, so the argument above
  implies the equality $h^{0,1}(X)=h^{1,0}(X)$.
\end{proof}

\begin{rmk}
By a result in progress of Conrad--Gabber we may generalize this Theorem
  to the situation where $K$ is an arbitrary non-archimedean field extension of
  \(\mathbf{Q}_p\). Indeed, let \(\mathcal{X}\) be a formal model of \(X\) with
  projective special fibre \(\mathcal{X}_0 \subset \mathbb{P}^n_k\). Then there
  exists \(\pi \in \mathfrak{m}\) such that \(p \in \pi \mathcal{O}_K\) and
  \(\mathcal{X}_{\pi} \coloneqq \mathcal{X} 
  \times_{\mathrm{Spf}(\mathcal{O}_K)} \mathrm{Spec}(\mathcal{O}_K/\pi)\)
  is also a projective variety over
  \(\mathrm{Spec}(\mathcal{O}_K/\pi)\)
  \[
    \xymatrix{
      \mathcal{X}_{\pi} \ar@{^{(}->}[rr] \ar[rd] & &
      \mathbb{P}^n_{\mathcal{O}_K/\pi} \ar[ld] \\
      & \mathrm{Spec}(\mathcal{O}_K/\pi). &
    }
  \]
  By a standard argument, there exists a finite type \(\mathbb{F}_p\)-algebra
  \(A\) with a morphism \(\phi \colon A \to \mathcal{O}_K/\pi\) and a diagram such
  that the diagram over \(\mathcal{O}_K/\pi\) is pullback along \(\phi\)
  \[
    \xymatrix{
      \mathcal{X}_{A} \ar@{^{(}->}[rr] \ar[rd] & &
      \mathbb{P}^n_A \ar[ld] \\
      & \mathrm{Spec}(A). &
    }
  \]
  Now the work of Conrad--Gabber would produce a proper flat family of formal
  schemes \(\mathcal{Y} \to \mathcal{Z}\) of topologically finite type over
  \(\mathbb{Z}_p\) whose reduction is a relative projective family
  \(\mathcal{X}_{U} \to U\) where \(U \subset \mathrm{Spec}(A)\) is a
  (non-empty) open and whose generic fibre \(Y \to Z\) is a proper smooth family
  of rigid spaces having \(X\) as one of the ``geometric fibres''. Applying
  Theorem~\ref{Main Corollary} to the family \(Y \to Z\) yields the equality of
  degree one hodge numbers of \(X\).
\end{rmk}

\section{The Albanese}

In this section we define another rigid group variety related to ``$1$-motives
of rigid spaces'', namely the rigid Albanese variety. We work in the slightly
more general setting where $X$ is a smooth proper rigid space over any complete
non-archimedean field $K$ of characteristic $0$ (not necessarily discretely
valued); as before, we fix a rational point $x: \Sp(K) \to X$. The only
non-formal input we require is the existence of the Picard variety associated
with $X$ in this generality, which is guaranteed by the work of Warner,
c.f.~\cite{Warner}.

\begin{defn}
  The rigid Albanese variety $(\mathcal{A},0)$ associated with $(X,x)$ is the
  initial object in the category of pointed maps from $(X,x)$ to an abeloid
  variety pointed at its origin.
\end{defn}

If no confusion seems likely, we call $\mathcal{A}$ the Albanese variety of $X$
and denote it by $\Alb_X$. If $\Alb_X$ exists, it is clearly unique up to
canonical isomorphism. In order to prove the existence of the Albanese, we
employ the Picard variety as follows:

\begin{construction}
  Let $\mathcal{A}$ be the dual of the maximal
  connected proper smooth subgroup of the Picard variety of $X$.\footnote{In our
    situation where the ground field has characteristic $0$, one can drop the
    smoothness in this definition.}
\end{construction}

Note that the maximal connected proper smooth subgroup of any commutative rigid
analytic group is well-defined; this is an easy exercise which we leave to the
reader. 
The existence and properties of the dual of an abeloid variety is provided
by~\cite[Corollary II.2]{abeloid}.

\begin{prop}
  The abeloid variety $\mathcal{A}$ constructed above is the Albanese of $X$.
\end{prop}

This Proposition is of no surprise, the corresponding versions in the scheme
case is well-known and can be found in~\cite[Theorem 3.3(iii)]{TDTE}. The proof
we give below is adapted from that in loc.~cit.

\begin{proof}
  To see that $\mathcal{A}$ has the correct universal property, note that the
  Poincar\'{e} bundle on $X \times \Pic^0_X$ restricts to a line bundle on $X
  \times \widehat{\mathcal{A}}$. Therefore we have a morphism $\Alb: X \to
  \mathcal{A}$. As the Poincar\'{e} bundle is trivialized along $\{x\} \times
  \Pic^0_X$, we know that $\Alb(x)=0$. Now, given any pointed morphism $\phi:
  (X,x) \to (\mathcal{A}',0)$, we may consider the line bundle ${(\phi \times
    id_{\widehat{\mathcal{A}'}})}^*\mathcal{L}$ on $X \times
  \widehat{\mathcal{A}'}$, the pullback of the Poincar\'{e} bundle $\mathcal{L}$
  on $\mathcal{A}' \times \widehat{\mathcal{A}'}$, which gives rise to a
  morphism $\hat{\phi}: \widehat{\mathcal{A}'} \to \Pic^0_X$. Since
  $\widehat{\mathcal{A}'}$ is proper and smooth, this morphism necessarily lands
  in $\widehat{\mathcal{A}}$, hence gives rise to $\hat{\phi}:
  \widehat{\mathcal{A}'} \to \widehat{\mathcal{A}}$. The dual of this morphism
  gives rise to a homomorphism $\widetilde{\phi}: \mathcal{A} \to \mathcal{A}'$.
  Using functoriality of Picard and (double)-duality of abeloid varieties,
  c.f.~\cite[Section 6.3]{Lut16}, we see that $\widetilde{\phi}$ as constructed
  above is canonical and $\phi=\widetilde{\phi} \circ \Alb$. This completes the
  proof.
\end{proof}

The Albanese property implies that the induced map between the first \'{e}tale
cohomology groups is injective. Before stating the result, recall that for a
rigid space $X$ over a non-archimedean field $K$ and any \'{e}tale sheaf
$\mathcal{F}$ on $X_{\'{e}t}$, the ``geometric'' \'{e}tale cohomology is defined
by $H^i_{\et}(X_{\bar{K}},\mathcal{F}) \coloneqq \varinjlim
H^1_{\et}(X_{L},\mathcal{F})$. It is a theorem of de Jong--van der Put that if
$X$ is quasi-compact, then we have $H^i_{\et}(X_{\bar{K}},\mathcal{F}) \cong
H^i_{\et}(X_{\mathbb{C}_K},\mathcal{F})$ as
$\mathrm{Gal}(K^{\mathrm{sep}}/K)$-modules (c.f.~\cite[Lemma 3.7.1 and Theorem
3.7.3]{dJvdP}).

\begin{prop}
\label{injectivity}
For any prime $l$ (which can be taken to be $p$), the natural map $\Alb^*:
H^1_{\et}(\Alb_{X,\bar{K}},\mathbb{Z}_l) \to
H^1_{\et}(X_{\bar{K}},\mathbb{Z}_l)$ is injective, and similarly for
$\mathbb{Q}_l$-coefficients.
\end{prop}

\begin{proof}
  It suffices to show the injectivity for $\mathbb{F}_l$-coefficients. An
  element $\xi \in H^1_{\et}(\Alb_{X,\bar{K}},\mathbb{F}_l)$ is represented by
  an \'{e}tale $\mathbb{F}_l$-torsor $\mathcal{B}$ over $\Alb_{X,L}$ where $L$
  is a finite separable extension of $K$. In this situation $\mathcal{B}$ itself
  is automatically an abeloid variety (c.f.~\cite[p.~167]{abelian}\footnote{The
    proof of the analogous result for abelian varieties given in
    loc.~cit.~extends with almost no change to the setting of abeloid varieties,
    except that one has to use the rigid geometry version of rigidity lemma
    (c.f.~\cite[Lemma 7.1.2]{Lut16})}). Choose any class $\xi$ such that $\Alb^*
  \xi=0$, in which case $X \times_{\Alb_X} \mathcal{B}=\mathcal{B}'$ is a
  trivial $\mathbb{F}_l$-torsor over $X$ (possibly after passing to a finite
  extension of $K$; from now on we will ignore the issue of base change and the
  reader should think of every statement as potentially true). In particular, we
  can choose a section $\sigma: X \to \mathcal{B}'$ to the natural projection,
  as in the following diagram:
  \[
    \xymatrix{
      \mathcal{B}'=X \times_{\Alb_X} \mathcal{B} \ar[r] \ar[d] & \mathcal{B} \ar[d] \\
      X \ar[r] \ar@{-->}@/^/[u] \ar@{-->}[ur] & \Alb_X \ar@{-->}@/^/[u] \\
    }
  \]
  The section $\sigma$ gives rise to a morphism from $X$ to $\mathcal{B}$ which
  can be chosen so that $x$ is sent to $0$. By the universal property of the
  Albanese we then get a section $\widetilde{\sigma}: \Alb_X \to \mathcal{B}$;
  but this just means that $\xi=0$, as desired.
\end{proof}

If the residue field $\kappa$ of \(K\) is of characteristic $0$, then
by~\cite[Theorem 1.18]{semistable}\footnote{The authors would like to thank
  Professor Michael Temkin for pointing this reference to us in a private
  communication.} $X$ is potentially semistable. Therefore the discussion in
Section~\ref{Preliminaries} applies automatically; in particular, by
Theorem~\ref{structure of picard} (3) we know that $\widehat{\Alb_X}$ has
dimension no bigger than that of the abeloid variety $A$ which appeared in the
aforesaid Theorem.

On the other hand, if $\kappa$ is of characteristic $p$, then we have no free
access to a structure theorem for the Picard variety anymore. Nevertheless, if
$K$ is a $p$-adic field we can still prove that the dimension of the Albanese is
no bigger than $h^{1,0}(X)$ by combining Proposition~\ref{injectivity} with a
little $p$-adic Hodge theory.

\begin{prop}
  If $X$ is a smooth proper rigid space over a $p$-adic field $K$, then we have
  \[
    \dim(\Alb_X) \leq \dim H^0(X,\Omega^1_X).
  \]
\end{prop}

\begin{proof}
  Proposition~\ref{injectivity} implies that the dimension of the Hodge--Tate
  weight $1$ (we follow the convention that $\mathbb{Q}_p(1)$ has Hodge--Tate
  weight $-1$) piece of $H^1_{\et}(\Alb_{X,\mathbb{C}_K},\mathbb{Q}_p)
  \otimes_{\mathbb{Q}_p} \mathbb{C}_K$ is at most that of
  $H^1_{\et}(X_{\mathbb{C}_K},\mathbb{Q}_p) \otimes_{\mathbb{Q}_p}
  \mathbb{C}_K$. One uses Hodge--Tate comparison again to see that the former is
  the dimension of $\Alb_X$ and the latter is $\dim H^0(X,\Omega^1_X)$
\end{proof}

Hoping that the map from $X$ to its Albanese could be given in terms of
``integration of $1$-forms on $X$'', one might na\"{i}vely speculate that the
dimension of $\Alb_X$ coincides with $h^{1,0}(X)$. However, we will see in the
next section that this fails in general; see Example~\ref{example} for an
explicit counterexample.

\section{Examples}

\begin{example}
  Let $A$ be an abeloid variety of dimension $d$ over a discretely valued
  non-archimedean field. Its Picard variety is an abeloid variety of the same
  dimension $d$. Then we have $h^{1,0}=h^{0,1}=d$ and the Albanese of $A$ is of
  course $A$ itself. The behavior of abeloid varieties is basically the same as
  abelian varieties according to~\cite[Theorem II]{abeloid}.
\end{example}

\begin{example}
\label{Hopf}
Let $H$ be a non-archimedean Hopf variety (c.f.~\cite{Hopfs} and~\cite{Hopfv}).
Then its Picard variety is $\mathbb{G}_m$. We have $h^{1,0}=0$ and $h^{0,1}=1$.
Since there is no non-constant morphism from a proper rigid variety to
$\mathbb{G}_m$ we see that the Albanese of $H$ is trivial.

There is a geometric explanation for why Hopf varieties should have trivial
Albanese.

\begin{prop}
  Let $K$ be a discretely valued non-archimedean field. There is no non-constant
  map from $\mathbb{A}^{1,\mathrm{rig}}_K$ to any Abeloid variety $A$.\footnote{This Proposition was essentially proved by W. Cherry in~\cite{W94} with the same proof (he only mentions abelian varieties instead of general abeloids, but the proof is identical in the abeloid case). We thank the referee for pointing this out to us.}
\end{prop}

\begin{proof}
  Applying~\cite[Theorem II]{abeloid} (after possibly passing to a finite
  separable extension of $K$), using the notations in loc.~cit., we may assume
  that \(A\) has a topological covering given by a smooth rigid group \(E\)
  sitting in an exact sequence of smooth rigid groups:
  \[
    0 \to T \to E \to B \to 0.
  \]
  Here \(T\) is a finite product of copies of \(\mathbb{G}_m^{\mathrm{rig}}\),
  and \(B\) is the generic fiber of a formal abelian scheme \(\mathcal{B}\) over
  \(\mathcal{O}_K\). By~\cite[Section 2]{abeloid} \(\mathcal{B}\) is the
  N\'{e}ron Model of \(B\). In particular, for any admissible smooth formal
  scheme \(\mathcal{X}\) (whose generic fiber is denoted as \(X\)) over
  \(\mathcal{O}_K\), any morphism from \(X\) to \(B\) extends uniquely to
  a morphism from \(\mathcal{X}\) to \(\mathcal{B}\). We would like to make the
  following:

  \begin{claim}
    Any map $\mathbb{A}^{1,\mathrm{rig}} \to A$ can be lifted to
    $\mathbb{A}^{1,\mathrm{rig}} \to E$.
  \end{claim}

  \begin{proof}
    By~\cite[Corollary 6.3.4]{Lut16}, it suffices to prove that
    $H^1(\mathbb{A}^{1,\mathrm{rig}},\mathbb{Z}) = 0$. One checks easily that
    the sheaf $\mathbb{Z}$ is overconvergent (c.f.~\cite[Definition after Lemma
    18]{Schneider}\footnote{Note that the concept of overconvergent is called
      conservative in loc.~cit.}). By~\cite[Corollary 20 (ii)]{Schneider}, it
    suffices to show that the associated Berkovich space of
    $\mathbb{A}^{1,\mathrm{rig}}$, denoted as
    $\mathcal{M}(\mathbb{A}^{1,\mathrm{rig}})$ in loc.~cit. (see
    also~\cite[Corollary 7.1.11]{FvdP} which compares the associated Berkovich
    space constructed by Berkovich and the site constructed
    in~\cite{Schneider}), is simply connected. Actually a stronger statement is
    true, namely the associated Berkovich space is contractible, due
    to~\cite[Theorem 6.1.5]{Berkovich}.
  \end{proof}

  Now it suffices to show that any morphism from $\mathbb{A}^{1,\mathrm{rig}}$
  to $B$ or $\mathbb{G}_m^{\mathrm{rig}}$ must necessarily be constant. The
  latter being well known, we shall just prove the former.

  We claim that any map \(f\) from \(\mathbb{A}^{1,\mathrm{rig}}\) to \(B\) is
  trivial. To see this, choose an increasing nested sequence of closed discs
  $D_i \subset \mathbb{A}^{1,\mathrm{rig}}$ which admissibly cover
  $\mathbb{A}^{1,\mathrm{rig}}$, and view $f$ as the limit of its restrictions
  to the $D_i$'s. Now, closed discs have obvious smooth formal models with
  special fiber \(\mathbb{A}^1_{\kappa}\). By the N\'{e}ron mapping property,
  any map from a closed disc to \(B\) would extend to a map from such a smooth
  formal model to \(\mathcal{B}\). Looking at the map on special fibers we get a
  map from a rational variety to an abelian variety, and any such map must be
  constant. Therefore our map \(f\) has image contained in an affinoid subspace
  of $B$. By the rigid analogue of Liouville's theorem (see
  Lemma~\ref{Liouville}) \(f\) must be constant.
\end{proof}

In the argument above, we used the following rigid analytic analogue of
Liouville's theorem.

\begin{lemm}
\label{Liouville}
There is no non-constant morphism from the analytification of a $K$-variety to a
$K$-affinoid space.
\end{lemm}

\begin{proof}
  It suffices to prove the following:
  \begin{claim}
    Let \(X=\Spec R\) be an affine integral scheme of finite type over \(K\).
    Then every bounded analytic function on \(X^{\mathrm{rig}}\) is a constant.
  \end{claim}

  We achieve this in 2 steps.

  Step 1: First, suppose that $R=K[x_1,\ldots,x_n]$. We have to prove that every
  bounded analytic function on \(\mathbb{A}^{n,\mathrm{rig}}_K\) is a constant.
  Recall that \(\mathbb{A}^{n,\mathrm{rig}}_K\) is given by inductive limit of
  \[
    \Sp K \langle x_1,\ldots,x_n \rangle \hookrightarrow \Sp K \langle
    x_1,\ldots,x_n \rangle \hookrightarrow \cdots \hookrightarrow \Sp K \langle
    x_1,\ldots,x_n \rangle \hookrightarrow \cdots,
  \]
  so the set of analytic functions on \(\mathbb{A}^{n,\mathrm{rig}}_K\) is given
  by
  \[
    \bigcap_{k \in \mathbb{N}} K \langle p^k x_1,\ldots,p^k x_n
    \rangle=\left\{ \sum_{I}a_I x^I | \lim_{I\to\infty}a_I p^{-k|I|}=0 \text{
        for all } k \in
    \mathbb{N} \right\}.
  \]
  The boundedness of such a function translates to the existence of a constant
  $C>0$ such that
  \[
    |a_I p^{-k|I|}| \leq C, \text{for all } k \in \mathbb{N},
    I=(i_1,\ldots,i_n).
  \]
  Therefore we get that each coefficient \(a_I\) must be zero except for
  \(I=(0,\ldots,0)\), and our function is constant as desired.

  Step 2: Choose an arbitrary $R$ as in the claim. By Noether normalization, we
  may assume that \(R\) is a finite algebra over \(K[x_1,\ldots,x_n]\). We claim
  that we can even assume that \(R\) is the integral closure of
  \(K[x_1,\ldots,x_n]\) in \(\Frac(R)\) and that \(\Frac(R)/
  \Frac(K[x_1,\ldots,x_n])\) is Galois with Galois group \(G\). Indeed, we only
  have to worry when \(\mathrm{Char}(K) = p > 0\). In that situation, by
  possibly passing to a finite inseparable extension of \(K\), we may find \(n
  \in \mathbb{N}\) such that 
  \(S_n \coloneqq \Frac(R) \otimes_{\Frac(K[x_1,\ldots,x_n])}
  \Frac(K[x_1^{1/p^n},\ldots,x_n^{1/p^n}])\) is separable over
  \(\Frac(K[x_1^{1/p^n},\ldots,x_n^{1/p^n}])\). Then we may dominate \(R\) by
  \(R'\), its integral closure in the Galois closure of one of the components of
  \(\mathrm{Spec}(S_n)\) over
  \(\Frac(K[x_1^{1/p^n},\ldots,x_n^{1/p^n}])\). Now any bounded analytic
  function pullback on \(\mathrm{Spec}(R)^{\mathrm{rig}}\) to a bounded analytic
  function on \(\mathrm{Spec}(R')^{\mathrm{rig}}\). If the latter is locally
  constant, then so is the former since \(\mathrm{Spec}(R')^{\mathrm{rig}}\)
  surjects onto \(\mathrm{Spec}(R)^{\mathrm{rig}}\).

  Step 3: Now assume that \(R\) is the integral closure of \(K[x_1,\ldots,x_n]\)
  in \(\Frac(R)\) and that \(\Frac(R)\) over \(\Frac(K[x_1,\ldots,x_n])\) is Galois with
  Galois group \(G\). Let \(f\) be a bounded analytic function on
  \(X^{\mathrm{rig}}\), and consider the functions
  \[
    a_i=\sum_{S \subset G,|S|=i} \prod_{g \in S} g(f)
  \]
  where \(g(f)(x)=f(g(x))\). It is easy to see that \(a_i\)'s are
  \(G\)-invariant, hence they are analytic functions on
  \(\mathbb{A}^{n,\mathrm{rig}}_K\). They are bounded functions, so by Step 1
  they are constants. Since \(f\) satisfies the equation
  \[
    f^n-a_1f^{n-1}+\cdots+{(-1)}^n a_n=0,
  \]
  we then see that \(f\) is a (locally) constant function as well.
\end{proof}

\begin{cor}
  Let $K$ be a discretely valued non-archimedean field, and $X$ be an
  $\mathbb{A}^1$-connected rigid variety over $K$. Then the Albanese of $X$ is
  trivial.
\end{cor}

\end{example}

We illustrate the failure of ``integrating 1-forms'' through the following
example.

\begin{example}
\label{example}
Let $A$ be a simple abeloid variety of dimension $d$ over a non-archimedean
field $K$. Choose a non-torsion point $P \in A$. Let
$Y=(\mathbb{A}^2_K-\{(0,0)\}) \times A$. Consider a $\mathbb{Z}$-action on $Y$
given as dilation by some topologically nilpotent element belonging to
$K^{\times}$ on the first factor and translation by $P$ on the second factor.
This action is properly discontinuous. Take $X=Y/\mathbb{Z}$. Projection to the
first factor makes $X$ into an isotrivial family of abeloids over a Hopf surface
$H$. One can compute the Hodge numbers of $X$ via Leray spectral sequence
applied to this projection. For example, we have $h^{1,0}(X)=d$ and
$h^{0,1}(X)=d+1$. But we make the following

\begin{claim}
  The Albanese of $X$ is trivial, i.e.~there is no nontrivial map from $X$ to
  any non-zero abeloid variety.
\end{claim}

\begin{proof}
  By our construction it suffices to show that there is no abeloid variety
  embedded as a subgroup of $\Pic^0_X$. The fibration 
  \[
    \xymatrix{
      A \ar[d] \\
      X \ar[r] & H \\
    }
  \]
  gives an exact sequence
  \[
    0 \to \mathbb{G}_m \to \Pic^0_X \to \hat{A} \to 0
  \]
  which exhibits $\Pic^0_X$ as the complement of the zero locus in the total
  space of a translation invariant line bundle $\mathcal{L}$ on $\hat{A}$
  (c.f.~\cite[Section 6.1]{Lut16}). Moreover, this translation invariant line
  bundle corresponds exactly to $P \in \hat{\hat{A}}=A$. Now, a morphism from an
  abeloid variety $\mathcal{A}$ to $\Pic^0_X$ is equivalent to the data of a
  homomorphism $f: \mathcal{A} \to \hat{A}$ and an isomorphism $s:
  \cO_{\mathcal{A}} \to f^*\mathcal{L}$. But since $\hat{A}$ is simple and
  $\mathcal{L}$ is non-torsion, such a morphism must be $0$. Therefore we
  conclude that there is no nontrivial connected proper subgroup in $\Pic^0_X$.
\end{proof}

An alternative argument due to Johan de Jong demonstrates that there is no
non-constant morphism from $Y$ (coming from $X$) to an abeloid variety. Indeed,
one notices that there is no non-constant morphism from a Hopf surface to an
abeloid variety (c.f. Example~\ref{Hopf}). Therefore any morphism $Y \to
\mathcal{A}$ must factor through $A$. But since such a morphism comes from $X$
it has to be invariant under translation by $P$. Thus we conclude that the
Albanese of $X$ above is trivial.
\end{example}

\section*{Acknowledgment}
The second named author would like to thank his advisor Professor Johan de Jong
for many helpful discussions during the preparation of this paper. He would also
like to thank Zijian Yao, Dingxin Zhang and Yang Zhou for discussing things
related to this paper.

We thank various anonymous referees heartily for providing many valuable
comments and suggestions concerning previous drafts of this article.


\end{document}